\documentclass[a4paper]{article}
\usepackage[T1]{fontenc}
\usepackage[utf8x]{inputenc}
\usepackage{url}
\usepackage{amsmath}
\usepackage{amssymb}
\usepackage{amsthm}
\usepackage[all]{xy}
\usepackage{hyperref}

\makeatletter
\newtheorem{theorem}{Theorem}
\newtheorem{lemma}{Lemma}
\newtheorem{definition}{Definition}
\newtheorem{proposition}{Proposition}
\newtheorem{remark}{Remark}

\pdfpageheight\paperheight
\pdfpagewidth\paperwidth

\newcommand{\email}[1]{\textbf{e-mail:~}\texttt{#1}}
\renewcommand{\vec}[1]{\mathbf{#1}}
\newcommand{\CC}{\mathbb{C}}
\newcommand{\m}{\mathfrak{m}}
\newcommand{\RR}{\mathbb{R}}
\newcommand{\M}{\mathcal{M}}
\newcommand{\C}{\mathcal{C}}

\newcommand{\K}{\mathbf{K}}

\newcommand{\ZZ}{\mathbb{Z}}
\newcommand{\TT}{\mathbb{T}}
\newcommand{\PG}{\mathbb{P}G}
\newcommand{\tst}{T^{*}S}

\newcommand{\PP}{\mathbb{P}}
\newcommand{\q}{\vec{q}}
\newcommand{\p}{\vec{p}}
\newcommand{\A}{\mathbf{A}}
\newcommand{\LL}{\mathbf{L}}
\newcommand{\jac}{\text{\ensuremath{\mathrm{Jac}}}}
\newenvironment{acknowledgements}{\subsubsection*{Acknowledgments}\it}{\rm}
\newcommand{\keywords}{\textbf{Keywords:} } 

\usepackage[all]{xy}

\begin{document}

\title{Algebro-geometric aspects of superintegrability: the degenerate Neumann
system}

\author{Martin Vuk\footnote{University of Ljubljana, Faculty of Computer and Information Science\newline \email{martin.vuk@fri.uni-lj.si}}}
\maketitle
\begin{abstract}
In this article we use algebro-geometric tools to describe the structure
of a superintegrable system. We study degenerate Neumann system with
potential matrix that has some eigenvalues of multiplicity greater
than one. We show that the degenerate Neumann system is superintegrable
if and only if its spectral curve is reducible and that its flow can
be linearized on the generalized Jacobians of the spectral curve.
We also show that the generalized Jacobians of the hypereliptic component
of the spectral curves are models for the minimal invariant tori of
the flow. Moreover the spectral invariants generate local actions
that span the invariant tori, while the moment maps for the rotational
symmetries provide additional first integrals of the system. Using
our results we reproduce already known facts that the degenerate Neumann
system is superintegrable, if its potencial matrix has eigenvalues
of multiplicity greater or equal than $3$.

\end{abstract}
\textbf{Submited to:} \emph{Communications in Mathematical Physics}
\section{Introduction}

The Neumann system describes the motion of a particle constrained
to a $n$-dimensional sphere $S^{n}$ in a quadratic potential. The
potential is in ambient coordinates $\q=(q_{1},\ldots,q_{n+1})\in\RR^{n+1}$
given as a quadratic form
\[
V(\q)=\frac{1}{2}\langle\A\q,\q\rangle=\frac{1}{2}\sum_{i=1}^{n+1}a_{i}q_{i}^{2}
\]
with the potential matrix $\A=\mathrm{diag}(a_{1},\ldots a_{n+1})$
and $a_{i}>0$. We study the \emph{degenerate} case when some of the
eigenvalues of $\A$ are of multiplicity greater than 3.

In the generic case where all the eigenvalues of the potential $a_{i}$
are different, the Neumann system is algebraically completely integrable
and its flow can be linearized on the Jacobian torus of an algebraic
spectral curve \cite{Moser:CSY:1980,Adler:AdvM:1980,Mumford:TataII:1984}. 

In contrast to the generic Neumann system, the degenerate case has
not received much attention. The Liouville integrability of degenerate
Neumann case was proven by Liu in \cite{Liu1992}. We have
used the methods of algebraic geometry to study the non superintegrable
case when multiplicity of the eigenvalues is less or equal 2 \cite{Vuk:JPA:2008}.
The reduction of the degenerate Neumann system to the Rosochatious
system was recently explored by Dullin and Han{\ss}mann in \cite{Dullin2012}. 

In this paper we study the degenerate case with the multiplicity of
some eigenvalues $\ge3$. We show that the degenerate Neumann system
is superintegrable if the spectral curve is reducible and that its
flow can be linearized on the generalized Jacobian of a singular hypereliptic
component of the spectral curve. We show that the generalized Jacobian
of the hypereliptic component is a model for minimal invariant tori
of the flow, which are of smaller dimension in superintegrable case.
Another difference between generic and superintegrable case is that
the spectral invariants do not generate all of the first integrals,
but only the actions that generate the invariant tori. 

To linearize the flow, we will use $(n+1)\times(n+1)$ Lax equation,
where the resulting spectral curve is singular in the degenerate case.
Following the standard procedure and normalizing the spectral curve
results in the loss of several degrees of freedom and in fact corresponds
to the symplectic reduction by the additional symmetries in the degenerate
case\cite{Vuk:JPA:2008,Dullin2012}. In order to avoid this, we will
use the generalized Jacobian of the singular spectral curve to linearize
the flow as in \cite{Gavrilov:MathZft:2002,Vivolo:EdMathSoc:2000,Vuk:JPA:2008}.
The generalized Jacobian is an extension of the ``ordinary'' Jacobian
by a commutative algebraic group (see \cite{Serre:Springer:ClassFields}
for more detailed description). 

One should also mention that apart from classical case there has been
a lot of interest in quantum case \cite{Macfarlane:NPB:1992,Babelon:NPB:1992,Harnad:CMP:1995}
for both Neumann and its reduction Rosochatius system. Our study should
provide additional insight into the superintegrable nature of quantum
degenerate Neumann system as well.

After the introduction, Lax representation and spectral curve of the
degenerate Neumann system are discussed in section \ref{sec:Degenerate-C.-Neumann}.
In subsection  \ref{sec:Invariants-of-the} we study invariant manifolds
of the isospectral flow. Our main result is formulated in theorem
\ref{thm:main} and states that the minimal invariant tori of degenerate
Neumann system are isomorphic to the real parts of the generalized
Jacobian of the singular hypereliptic component of the spectral curve
modulo some discrete group action.

\section{Degenerate Neumann system\label{sec:Degenerate-C.-Neumann}}

The Neumann system is a Hamiltonian system on the cotangent bundle
of the sphere $T^{*}S^{n}$ with the standard symplectic form. The
Hamiltonian can be written in ambient coordinates $(\q,\p)\in T^{*}S^{n}=\{(\q,\p)\in\RR^{n+1}\times\RR^{n+1};\quad\langle\q,\q\rangle=1\wedge\langle\p,\q\rangle=0\}$
as
\[
H=\frac{1}{2}\langle\p,\p\rangle+V(\q)=\frac{1}{2}\sum_{i=1}^{n+1}p_{i}^{2}+\frac{1}{2}\sum_{i=1}^{n+1}a_{i}q_{i}^{2}
\]
and the equations of motion can be written in Hamiltonian form
\begin{eqnarray}
\dot{\q} & = & \p\label{eq:hamiltonian-system}\\
\dot{\p} & = & -\A\q+\varepsilon\q\nonumber 
\end{eqnarray}
 where $\varepsilon=\left\Vert \p\right\Vert ^{2}+c$ is chosen so
that $\left\Vert \q\right\Vert =1$ and the particle stays on the
sphere.
\begin{definition}
We will call the Neumann system \emph{degenerate} if some of the eigenvalues
of $\A$ are of multiplicity grater than 1. 
\end{definition}
We assume throughout this paper that $\A$ has $k$ different eigenvalues
\[
\{a_{1},\ldots a_{k}\}
\]
 with multiplicities $\{m_{1},\ldots m_{k}\}$ and we denote by $r$
the number of eigenvalues with multiplicity grater than one. 

It is known that degenerate Neumann system is superintegrable if $m_{j}\ge3$
for some $j$ and is not superintegrable if all $m_{j}\le2$. Its
structure is well described in \cite{Vuk:JPA:2008} for the case of
all $m_{j}\le2$ and in \cite{Dullin2012} for the general case.

\subsection{Isospectral flow and spectral curve}

A standard approach to study integrable systems is by writing down
the system in the form of Lax equation, which describes the flow of
matrices or matrix polynomials with constant eigenvalues, i.e. the
isospectral flow \cite{Adler:AdvM:1980,audin:CUP:1999}. Eigenvalues
of the isospectral flow are the first integrals of the integrable
system and Lax representation maps Arnold-Liouville tori into the
real part of the isospectral manifold consisting of matrices with
the same spectrum. A quotient of the isospectral manifold by a suitable
gauge group is in turn isomorphic to the open subset of the Jacobian
of the spectral curve \cite{Beauville:Acta:1990}. 

We extend this approach to cover special properties of the superintegrabile
systems. The isospectral flow for Neumann system can be written in
the form of a differential equation on the space of matrix polynomials
of degree $2$
\begin{equation}
\frac{d}{dt}\LL(\lambda;t)=[\LL(\lambda;t),(\lambda^{-1}\LL(\lambda;t))_{+}]\label{eq:Lax}
\end{equation}
and with the initial condition of the special form
\begin{equation}
\LL_{0}(\lambda)=\lambda^{2}\A+\lambda\q\wedge\p-\q\otimes\p\label{eq:Lax_form}
\end{equation}
for $\q,\p\in\RR^{n+1}$ such that $||\q||=1$ and $\langle\q,\p\rangle=0$.
\begin{remark}
Two different Lax equations are known for a generic Neumann system
with $n$ degrees of freedom: one is using $(n+1)\times(n+1)$ matrix
polynomials of degree $2$ \cite{Adler:AdvM:1980} and the other is
using $2\times2$ matrix polynomials of degree $n$ \cite{Mumford:TataII:1984,Beauville:Acta:1990}.
Both of the Lax representations can be used to describe degenerate
case, the $(n+1)\times(n+1)$ Lax equation make it easier to investigate
the superintegrable properties of the system.\end{remark}
\begin{lemma}
\label{lem:The-solution-of}The solution of (\ref{eq:Lax})
with initial condition of the form (\ref{eq:Lax_form}) is of the
same form 
\begin{equation}
\LL(\lambda;t)=\lambda^{2}\A+\lambda\q(t)\wedge\p(t)-\q(t)\otimes\q(t)\label{eq:Lax_form_t}
\end{equation}
for all time $t\in\RR$. Moreover the vectors $\q(t),\p(t)\in\RR^{n+1}$
satisfy $||\q(t)||=1$ and $\langle\q(t),\p(t)\rangle=0$ and satisfy
Hamiltonian equations for Neumann system.
\end{lemma}
\begin{proof}
We will show, that the path of the form (\ref{eq:Lax_form_t}) can
satisfy the Lax equation if $(\q(t),\p(t))$ are the solutions of
the Neumann equations. Lemma follows from the uniqueness for the solution
of the ordinary differential equations. Let's put a curve of the form
(\ref{eq:Lax_form_t}) into the Lax equation 
\begin{eqnarray*}
\frac{d}{dt}(\lambda^{2}\A+\lambda\q\wedge\p-\q\otimes\q) & = & [\lambda^{2}\A+\lambda\q\wedge\p-\q\otimes\q,\lambda\A+\q\wedge\p]\\
 & = & 0\cdot\lambda^{3}+0\cdot\lambda^{2}-[\q\otimes\q,\A]\lambda-[\q\otimes\q,\q\wedge\p].
\end{eqnarray*}
and compare the terms at powers of $\lambda$
\begin{eqnarray*}
\frac{d\A}{dt} & = & 0\\
\frac{d\q\wedge\p}{dt} & = & -[\q\otimes\q,\A]\\
\frac{d\q\otimes\q}{dt} & = & -[\q\otimes\q,\q\wedge\p].
\end{eqnarray*}
The leading coefficient $\A$ does not change and one can see directly
that the second and third equations are satisfied if $\q(t),\p(t)$
are solutions to the Neumann problem. From the uniqueness theorem
for ordinary differential equations it follows that the isospectral
flow $\LL(\lambda;t)$ is always of the form(\ref{eq:Lax_form_t}).
\end{proof}
Once we have an isospectral flow, its invariants can be described
in terms of the characteristic polynomial and isospectral subsets
of matrix polynomials. Indeed if the flow is isospectral, the characteristic
polynomial 
\[
P(\lambda,\mu)=\det(\LL(\lambda)-\mu)
\]
does not change along the flow. The characteristic polynomial is therefore
a conserved object of the flow (its coefficients and also zeros are
conserved quantities). The \emph{affine spectral curve} of the flow
(\ref{eq:Lax}) is defined as the set of points $(\lambda,\mu)\in\CC\times\CC$
that satisfy the characteristic equation
\[
\det(\LL(\lambda)-\mu)=0
\]
and is also conserved along the flow. We can complete the affine spectral
curve to obtain a compact Riemann surface. To obtain maximal smoothness
the affine spectral curve is embedded into the surface $\mathcal{O}(2)$
given by two patches of $\CC^{2}$ glued together by the coordinate
change
\[
(\lambda,\mu)\mapsto\left(\frac{1}{\lambda},\frac{\mu}{\lambda^{2}}\right).
\]
We denote the latter coordinate by $z=\frac{\mu}{\lambda^{2}}$.
\begin{definition}
\emph{The spectral curve }of the flow (\ref{eq:Lax}) is the compactification
of the affine spectral curve in the surface $\mathcal{O}(2)$. We
denote it by $\C$. 
\end{definition}
There is an important relation between singularities of the spectral
curve and regularity of the matrix $\LL(\lambda)$. 
\begin{definition}
We will call a matrix $\A\in\CC^{n\times n}$ \emph{regular }if all
of its eigenvalues have one-dimensional eigenspace. Equivalently $\A$
is regular if all the eigenvalues have the geometric multiplicity
equal to one.\emph{ }
\end{definition}
It is easy to see, that if the matrix $\LL(\lambda_{0})$ is not regular,
then the curve $\C$ has a singularity in at least one of the points
of $\lambda^{-1}(\lambda_{0})$. More precisely if $\mu_{0}$ is an
eigenvalue of $\LL(\lambda_{0})$ with eigenspace of dimension grater
than one, then the point $(\lambda_{0},\mu_{0})\in\C$ is singular
(see \cite{Gavrilov:MathZft:2002} for proof).

\subsection{Spectral curve for degenerate Neumann system }

As the important singularities occur at $\infty$, we will calculate
the affine curve in coordinates $(\lambda,z)=(\lambda,\frac{\mu}{\lambda^{2}})$.
The Lax matrix at infinity is given as a rank 2 perturbation of the
matrix $\A$
\[
\lambda^{-2}\LL(\lambda)=\A+\lambda^{-1}\q\wedge\p-\lambda^{-2}\q\otimes\q.
\]
To obtain the equation of the spectral curve, we follow \cite{Moser:CSY:1980,Liu1992}
and use Weinstein-Aronszajn determinant as follows 
\[
\frac{\det(z-\lambda^{-2}\LL(\lambda))}{\det(z-\A)}=1-\Phi_{z}(\q,\p),
\]
where $\Phi_{z}(\q,\p)=-\lambda^{-2}Q_{z}(\q)-\lambda^{-2}(Q_{z}(\q)Q_{z}(\p)-Q_{z}^{2}(\p,\q))$
and
\begin{eqnarray*}
Q_{z}(\vec{x},\vec{y}) & = & \langle(z-\A)^{-1}\vec{x},\vec{y}\rangle\\
Q_{z}(\vec{x}) & = & Q_{z}(\vec{x},\vec{x})
\end{eqnarray*}
for any $\vec{x},\vec{y}\in\mathbb{R}^{n+1}$. If we write $\Phi_{z}$
in co-ordinates
\begin{eqnarray*}
-\lambda^{2}\Phi_{z} & = & Q_{z}(\q)+Q(\q)Q_{z}(\p)-Q_{z}^{2}(\p,\q)\\
 & = & \sum\frac{q_{i}^{2}}{z-a_{i}}+\sum\frac{q_{i}^{2}}{z-a_{i}}\sum\frac{p_{i}^{2}}{z-a_{i}}-\left(\sum\frac{p_{i}q_{i}}{z-a_{i}}\right)^{2}
\end{eqnarray*}
 we can write $\det(z-\lambda^{-2}\LL(\lambda))=\det(z-\A)(1-\Phi_{z})$
and 
\begin{multline*}
\det(\mu-\LL)=\lambda^{2(n+1)}\det(z-\lambda^{-2}\LL)=\\
=\lambda^{2n+2}\prod(z-a_{i})\left(1+\lambda^{-2}(Q_{z}(\q)-Q_{z}(\q)Q_{z}(\p)+Q_{z}^{2}(\q,\p))\right)\\
=\lambda^{2n}\prod(z-a_{i})\left(\lambda^{2}+Q_{z}(\q)-Q_{z}(\q)Q_{z}(\p)+Q_{z}^{2}(\q,\p)\right).
\end{multline*}
Let us assume, that $\A$ is diagonal and has eigenvalues $a_{i}$,
$i=1,\ldots,k$ with multiplicities $m_{i}\ge1$. Denote by $V_{i}\subset\{1,\ldots,n+1\}$
a subset of indexes for which the diagonal element of $\A$ is $a_{i}$.
Denote by 
\[
||\q||_{i}^{2}=\sum_{j\in V_{j}}q_{j}^{2},\quad\langle\q,\p\rangle_{i}=\sum_{j\in V_{j}}q_{j}p_{j}\text{ and }||\p||_{i}^{2}=\sum_{j\in V_{j}}p_{j}^{2}.
\]
Then $\Phi_{z}$ can be written as
\begin{eqnarray*}
-\lambda^{2}\Phi_{z}(\q,\p) & = & \sum_{i=1}^{k}\frac{||\q||_{i}^{2}}{z-a_{i}}+\sum_{i=1}^{k}\frac{||\q||_{i}^{2}}{z-a_{i}}\sum_{i=1}^{k}\frac{||\p||_{i}^{2}}{z-a_{i}}-\left(\sum_{i=1}^{k}\frac{\langle\q,\p\rangle_{i}^{2}}{z-a_{i}}\right)^{2}\\
 & = & \sum_{i=1}^{k}\frac{F_{i}}{z-a_{i}}+\sum_{m_{j}>1}\frac{K_{j}^{2}}{(z-a_{j})^{2}}\\
 & = & \frac{Q_{\q,\p}(z)}{\prod_{i,m_{i}=1}(z-a_{i})\prod_{j,m_{j}>1}(z-a_{j})^{2}}
\end{eqnarray*}
for a polynomial $Q_{\q,\p}(z)$ of degree $k-1+r$, where $r$ is
the number of eigenvalues with multiplicity $>1$. If the multiplicity
of an eigenvalue $a_{i}$ is 1, the terms with $(z-a_{i})^{2}$ cancel,
if on the other hand, the multiplicity is greater than one, the quadratic
term remains. We have obtained a set of commuting first integrals
$F_{i}$ for $i=1,\ldots,k$ and $K_{j}^{2}$ for $m_{j}>1$ for the
degenerate Neumann system, which were first described in \cite{Liu1992}.
The integrals $F_{i}$ are known as Uhlenbeck's integrals for non-degenerate
case, while the integrals $K_{j}$ are total angular momenta 
\[
K_{j}=\sum_{v,u\in V_{j}}(q_{v}p_{u}-q_{u}p_{v})^{2}
\]
for the action of the symmetry group of rotations $O(m_{j})$ acting
on the eigenspace of $\A$ for the eigenvalue $a_{j}$. Moreover 
\[
\sum F_{i}=1
\]
and Hamiltonian $H$ can be written as
\[
H=\sum a_{i}F_{i}+\frac{1}{2}\sum K_{j}^{2}.
\]
The number of first integrals we have obtained from the spectral curve
is $k-1+r<n$ and is smaller than the number of degrees of freedom.
There are more integrals (arising from the actions of the orthogonal
group $O(m_{j})$), which do not appear in the algebraic picture.
Taking $\mu=\lambda^{2}z$, the spectral polynomial becomes
\begin{multline}
\det(\mu-\LL)=\lambda^{2n}\prod_{i=1}^{k}(z-a_{i})^{m_{i}}\left(\lambda^{2}+\frac{Q_{\q,\p}(z)}{\prod(z-a_{i})\prod(z-a_{j})^{2}}\right)\\
=\lambda^{2n}\prod_{i,m_{i}>1}(z-a_{i})^{m_{i}-2}\left(\lambda^{2}\prod_{m_{i}=1}(z-a_{i})\prod_{m_{i}>1}(z-a_{j})^{2}+Q_{\q,\p}(z)\right).\label{eq:spectral_poly}
\end{multline}

\begin{proposition}
The spectral curve $\C$ is an union of the following components\end{proposition}
\begin{enumerate}
\item a point $\lambda=0,\mu=0$
\item a collection of parabolas
\[
\mu=a_{i}\lambda^{2},
\]
coming from $(z-a_{i})=0$, for $i$ such that $m_{i}>1$.
\item a singular hypereliptic curve $\C_{h}$
\[
\lambda^{2}=\frac{Q_{\q,\p}(z)}{\prod_{i,m_{i}=1}(z-a_{i})\prod_{j,m_{j}>1}(z-a_{j})^{2}}
\]
with the arithmetic genus $g_{a}=k+r-1$.
\end{enumerate}
We see that all of the first integrals $F_{i}$ an $K_{j}$ coming
from the spectral curve are encoded in the hypereliptic part $\C_{h}$.
Also note, that the genus of the parabolas $\mu=a_{i}\lambda^{2}$
is $0$. This suggests that the flow of the system will be entirely
described by the hypereliptic curve $\C_{h}$ and its generalized
Jacobian. We will also show that the arithmetic genus $k+r-1$ gives
the dimension of the invariant tori we can expect for the Neumann
system. 

The desingularisation of the hypereliptic curve $\C_{h}$ is given
by the equation 
\[
w^{2}=\left(\lambda\prod_{i=1}^{k}(z-a_{i})\right)^{2}=Q_{\q,\p}(z)\prod_{i,m_{i}=1}(z-a_{i})
\]
and is of genus $g=[k-\frac{1}{2}m]-1$. We denote it by $\tilde{\C_{h}}$. 
\begin{remark}
It has been shown that the normalized spectral curve $\tilde{\C_{h}}$
can be used to describe the dynamics of the Rosochatius system, which
can be seen as a reduction of degenerate Neumann system by $SO(m_{j})$
symmetries \cite{Vuk:JPA:2008}. We will show later in this paper
that the singular hypereliptic curve $\C_{h}$ describes the dynamics
of the degenerate Neumann system and that the set of parabolas $\mu=a_{j}\lambda^{2}$
accounts for the loss of dimensions in the invariant tori as a result
of the superintegrability of the system. 
\end{remark}

\section{Invariants of the flow\label{sec:Invariants-of-the}}

We have seen, that the isospectral nature of the flow (\ref{eq:Lax})
reveals the first integrals of the system encoded as coefficients
in the hypereliptic part of the spectral curve. We have also observed,
that not all of the first integrals can be expressed in terms of the
spectral data. Next we would like to describe the invariant sets of
the isospectral flow (\ref{eq:Lax}) and hence the invariant tori
of the Neumann system. 

If an integrable system is not superintegrable, then the generic invariant
sets of its flow are Arnold-Liouville tori. In the case of superintegrability,
the generic invariant sets are still tori, but of lower dimension.
We will show, that invariant tori of the degenerate Neumann system
can be obtained naturally as real parts of the generalized Jacobian
of hypereliptic part $\C_{h}$ of the spectral curve. 

Both the characteristic polynomial as well as the spectral curve are
invariants of the flow, so the trajectories of (\ref{eq:Lax}) lie
inside the subspace of all the matrix polynomials $\LL(\lambda)$
with the same spectral polynomial (or in fact the same spectral curve).
\[
\M_{\C}=\{\LL(\lambda)=\A\lambda^{2}+{\bf B}\lambda+{\bf C};\quad\text{spectral curve of }\LL(\lambda)\text{ is }\C\}.
\]
The set $\M_{\C}$ is way too large to be the minimal invariant set.
We therefore look for additional invariants. The flow of \ref{eq:Lax}
conserves the leading term $\A$, so the set 
\[
\M_{\C,\A}=\{\LL(\lambda);\quad\LL\in\M_{\C}\text{ and }\LL(\infty)=\A\}
\]
is also an invariant of the flow. The set $\M_{\C,\A}$ is too large
to be minimal invariant torus. For the submatrices of $\LL(\lambda)$,
where $\A$ has a non-regular block, the second highest power of $\lambda$
is also preserved by the flow. Let us denote by $\K_{j}$ the sub-matrix
of ${\bf B}$ that correspond to eigenspace of $\A$ for eigenvalue
$a_{j}$ with multiplicity $m_{j}>1$. In case of the degenerate Neumann
system the matrix $\K_{j}$ is an antisymmetric matrix whose entries
are 
\[
k_{ik}=q_{v_{j}(i)}p_{v_{j}(k)}-q_{v_{j}(k)}p_{v_{j}(i)}
\]
for indexes $v_{j}(i),v_{j}(k)\in V_{j}$. The entries $k_{ik}$ are
angular momenta in $q{}_{v_{j}(i)},q_{v_{j}(k)}$ plane and are the
components of the momentum map for the $O(m_{j})$ action on the eigenspace
of $\A$ corresponding to eigenvalue $a_{j}$. In fact the block $\K_{j}\in\mathfrak{so}(m_{j})$
is the value of the momentum map itself. Since the Hamiltonian $H$
is preserved by those actions, the blocks $\K_{j}$ are also a conserved
quantity. The isospectral flow (\ref{eq:Lax}) will therefore fix
the blocks $\K_{j}$ for all the eigenvalues $a_{j}$ with multiplicity
grater than $1$. We can then restrict ourselves to the isospectral
set
\[
\M_{\C,\A,\K_{j}}=\{\LL(\lambda);\quad\LL(\lambda)\in\M_{\C},\A,\K_{j}\text{ are fixed}\}.
\]
By the result in \cite{Vivolo:EdMathSoc:2000} the isospectral manifold
$\M_{\C,\A,\K_{j}}$ is isomorphic as an algebraic manifold to the
Zariski open subset $J(\C_{m})-\Theta$ of the generalized Jacobian
of the singularization of the spectral curve $\C$, given by a modulus
\[
\m=\sum_{P\in\lambda^{-1}(\infty)}P
\]
on the normalization $\tilde{C}$ of the spectral curve. Furthermore
the following diagram commutes
\[
\xymatrix{0\ar[r] & \PG_{\A,\K_{j}}\ar[r]\ar[d] & \M_{\C,\A,\K_{j}}\ar[r]\ar[d]_{e_{m}} & \M_{\C,\A,\K_{j}}/\PG_{\A,\K_{j}}\ar[r]\ar[d]_{e} & 0\\
0\ar[r] & (\CC^{*})^{n}\ar[r] & \jac(\C_{\m})-\Theta\ar[r] & \jac(\tilde{\C})-\Theta\ar[r] & 0
}
\]
where the group $\PG_{\A,\K_{j}}$ is the subgroup of the projective
general linear group that leaves the matrix $\A$ and the sub-matrices
$\K_{j}$ invariant by conjugation. If $\K_{j}$ are regular, the
group $\PG_{\A,\K_{j}}$ is isomorphic to a complex torus $(\CC^{*})^{n}$,
since both $\A$ and $\K_{j}$ can be diagonalized. The stabilizer
$G_{\A,\K_{j}}$ is a product 
\[
G_{\A,\K_{j}}=\TT_{\A}\times\prod_{m_{j}\ge2}{\rm \PP stab\K_{j}}
\]
where $\TT_{\A}$ consists of diagonal matrices $[t_{ij}]$ such that
$t_{ll}=t_{kk}$ if the indexes $k$ and $l$ correspond to the same
eigenvalue $a_{j}$ of $\A$ ($k,l\in v_{j})$. The group of ${\rm stab}\K_{j}$
consist of all the matrices, that can be diagonalized with the same
coordinate change as $\K_{j}$. The subgroup $\TT_{j}=\left\{ \exp(t\K_{j});t\in\CC\right\} $
generated by $\K_{j}$ is a subgroup of the stabilizer ${\rm stab}\K_{j}$.
Denote by $\TT_{\A,\K_{j}}$ the group
\[
\TT_{\A,\K_{j}}=\TT_{\A}\times\prod_{m_{j}>2}\left(\PP{\rm stab}\K_{j}/\TT_{j}\right).
\]
The orbits of the group $\TT_{\A,\K_{j}}$ are transversal to the
isospectral flow (\ref{eq:Lax}), so the right space to consider as
the invariant set is the quotient $\M_{\C,\A,\K_{j}}/\TT_{\A,\K_{j}}$. 

On the other hand we have seen from lemma \ref{lem:The-solution-of},
that the set
\[
\mathcal{P}_{\A}=\{\A\lambda^{2}+\lambda\p\wedge\q-\q\otimes\q;\quad\q,\p\in\CC^{n+1}\times\CC^{n+1}\}
\]
is also invariant with respect to the flow (\ref{eq:Lax}). The map
$(\p,\q)\mapsto\A\lambda^{2}+\lambda\p\wedge\q-\q\otimes\q$ gives
a parametrization of the complexified phase space of the Neumann system
into the set of matrix polynomials. The appropriate invariant set
would therefore be the intersection $\mathcal{P}_{\A}\cap\M_{\C,\A,\K_{j}}$,
which we denote by $\M_{N}$. The following theorem describes the
relation among $\M_{N}$ and $\M_{\C,\A,\K_{j}}/\TT_{\A,\K_{j}}$.
\begin{theorem}
\label{thm:main} For any initial condition of the Neumann
system such that the normalized hypereliptic part $\tilde{\C}_{h}$
of the spectral curve is smooth the following holds. 
\begin{enumerate}
\item If $\K_{j}$ is regular, the quotient of $\M_{N}$ by a discrete group
is isomorphic as an algebraic variety to the open subset of the generalized
Jacobian $\jac(\C_{h})-\Theta\simeq\M_{\C,\A,\K_{j}}/\TT_{\A,\K_{j}}$
of the hypereliptic spectral curve $\C_{h}$.
\item If $\K_{j}=0$ ($K_{j}=0$) the set $\M_{N}$ is isomorphic as an
algebraic variety to the open subset $\jac(\tilde{\C}_{h})-\Theta\simeq\M_{\C,\A,\K_{j}}/\PG_{\A,\K_{j}}$
of the Jacobian of the desingularized hypereliptic curve for $\tilde{\C}_{h}$.
\item The flow generated by the Hamiltonian $H$ is linear on the generalized
Jacobian $\jac(\C_{h})$ if $\K_{j}$ regular or Jacobian $\jac(\tilde{\C_{h}})$
if $\K_{j}=0$.\label{enu:H-linear}
\end{enumerate}
\end{theorem}
\begin{remark}
The intermediate cases where $K_{j}\not=0$ and $\K_{j}$ are non-regular
appears for dimensions greater than $3$ and will not be covered in
this article. \end{remark}
\begin{proof}
We can define the eigenbundle map from the isospectral set $\M_{\C,\A,\K_{j}}$
to $\jac(\C_{h})$ 
\[
e:\M_{\C,\A,\K_{j}}\to\jac(\C_{h})
\]
by normalizing eigenline bundles on $\tilde{\C}_{h}$ as in \cite{Gavrilov:MathZft:2002,Vivolo:EdMathSoc:2000,Vuk:JPA:2008}.
The set $\M_{\C,\A,\K_{j}}$is isomorphic to an open subset of the
generalized Jacobian of the singular spectral curve $\C_{\m}$, defined
by the modulus 
\[
\m=\sum_{P\in\lambda^{-1}(\infty)}P
\]
on the reducible normalized spectral curve $\tilde{C}$ \cite{Vivolo:EdMathSoc:2000}.
We also see, that the following diagram commutes 
\[
\xymatrix{\M_{\C,\A,\K_{j}}\ar[r]\ar[d]_{e_{m}} & \M_{\C,\A,\K_{j}}/\TT_{\A,\K_{j}}\ar[r]\ar[d]_{e} & \M_{\C,\A,\K_{j}}/\PG_{\A,\K_{j}}\ar[d]_{\tilde{e}}\\
\jac(\C_{\m})-\Theta_{\m}\ar[r] & \jac(\C_{h})-\Theta_{h}\ar[r] & \jac(\tilde{\C_{h}})-\tilde{\Theta_{h}}
}
\]
where $\TT_{\A}=\PG_{\A,\K_{j}}/\prod_{m_{j}\ge2}\TT_{j}$ and $\TT_{j}=\left\{ \exp(t\K_{j})\right\} $
is a subgroup generated by the elements $\K_{j}\in\mathfrak{so}(m_{j})<\mathfrak{so}(n+1)$.

We have to prove, that the orbit of $\TT_{\A}$ $\M_{\C,\A,\K_{j}}/\TT_{\A}$
intersects the set $\M_{N}$ in finitely many points. 

Since $\K_{j}$ are anti-symmetric all of its eigenvalues are pure
imaginary and are of multiplicity one if $\K_{j}$ is a regular matrix,
the group $\PG_{\A,\K_{j}}$ is isomorphic to the torus $\TT_{\CC}^{r}=\left(\CC^{*}\right)^{n}$.
We only have to prove that the set $\M_{N}$ consisting of matrices
of the form (\ref{eq:Lax_form}) is a covering of the quotient $\M_{\C,\A,\K_{j}}/\TT_{\A,\K_{j}}$.
Or in other words (\ref{eq:Lax_form}) gives a parametrization of
the quotient $\M_{\C,\A,\K_{j}}/\TT_{\A,\K_{j}}$ by $(\q,\p)\in F^{-1}(f)\subset(\tst^{n})^{\CC}$.
First note that the map 
\[
J^{\A}:\M_{N}\to\M_{\C,\A,\K_{j}}
\]
 defined by (\ref{eq:Lax_form}) is an immersion. We will show that
any orbit of $\TT_{\A}$ intersects the image of $J^{\A}$ only in
finite number of points. To explain how $\PG_{\A,\K_{j}}$ acts on
the Lax matrix (\ref{eq:Lax_form}), note that an element $g\in\PG l(n+1)$
acts on a tensor product $x\otimes y$ of $x,y\in\CC^{n+1}$ 
\[
g:x\otimes y\mapsto(gx)\otimes((g^{-1})^{T}y)
\]
 by multiplying the first factor with $g$ and the second with $(g^{-1})^{T}$.
The subgroup of $\PG_{\A,\K_{j}}$, for which the generic orbit lies
in the image of $J^{\A}$, is given by orthogonal matrices 
\[
D:=O(n,\CC)\cap\PG_{\A,\K_{j}}\simeq(\ZZ_{2})^{n-r-1}\times\prod_{m_{j}>2}(\TT_{j}\cap O(m_{j},\CC)).
\]
 There are special points in the image of $J^{\A}$ that have a large
isotropy group (take for example $q_{i}=\delta_{ij}$ and $p_{i}=\delta_{ik}$,
$k\not=j$, where the isotropy is $(\CC^{*})^{n-2}$). But the intersection
of any orbit with the image of $J^{\A}$ coincides with the orbit
of $D$. If we take the torus $\TT_{\A,\K_{j}}<\PG_{\A,\K_{j}}$ the
orbits of $\TT_{\A,\K_{j}}$ will intersect image of $J^{\A}$ only
in the orbit of the finite subgroup $(\ZZ_{2})^{n-r-1}$, since $\TT_{\A,\K_{j}}\cap\TT_{j}=\{Id\}$.
We have proved that the level set of Lax matrices $\M_{N}=\left\{ \LL(\lambda)=J^{\A}(\q,\p)\right\} $
is an immersed sub-manifold in $\M_{\C,\A,\K_{j}}$ that intersects
the orbits of torus $\TT_{\A,\K_{j}}\simeq(\CC^{*})^{n-1}$ in only
finite number of points and is thus a covering of the quotient $\M_{\C,\A,\K_{j}}/\TT_{\A,\K_{j}}$.

To prove the assertion (\ref{enu:H-linear}), note that the matrix
polynomial $\mathbf{M}(\lambda)$ in the Lax equation (\ref{eq:Lax})
is given as a polynomial part $R(\lambda,\LL(\lambda))_{+}$ for a
polynomial $R(z,w)=zw$ and such isospectral flows are mapped by $e$
to linear flows on the Jacobian $\jac(\C)$ (see \cite{audin:CUP:1999}
for reference). 
\end{proof}
Taking into account the real structure on $\C$, the invariant tori
on the Hamiltonian flow generated by $H$ can be described as a real
part of the generalized Jacobian.
\begin{theorem}
\label{thm:real-ext}For any initial condition of the Neumann
system such that $\K_{j}$ are regular and $\tilde{\C}_{h}$ smooth
the following holds.
\begin{enumerate}
\item the invariant tori of the flow of $H$ are $(\ZZ_{2})^{n-r-1}$
coverings of the real part of the generalized Jacobian. 
\item The rotations generated by total angular momenta $K_{j}^{2}$
are precisely the rotations of the factors $S^{1}$ in the fiber $\mathbb{T}^{r}$
in the fibration $\mathbb{T}^{r}\to\jac(\C)^{\RR}\to\jac(\tilde{\C})^{\RR}$,
which is the real part of the fibration $(\CC^{*})^{r}\to\jac(\C_{m})\to\jac(\C)$.
\end{enumerate}
\end{theorem}
\begin{proof}
The groups $\TT_{j}$ defined in the proof of theorem \ref{thm:main}
are the complexification of the group of rotations in the eigenspace
of $\A$ for the eigenvalue $a_{j}$ that are generated by the Hamiltonian
vector fields of $K_{j}^{2}$. The flows of the total angular momenta
$K_{j}$ therefore generate the fibers $\CC^{*}$ in the extension
$(\CC^{*})^{r}\to\jac(\C_{h})\to\jac(\tilde{\C}_{h})$. 

To prove the claim about the real part, let us denote by $S_{j}$
singular points of $\C_{h}$ corresponding to the eigenvalue $a_{j}$
of the potential $\A$ with multiplicity $>1$. Since we assumed $\K_{j}$
is regular, two points $P_{j}^{+}$and $P_{j}^{-}$ on the smooth
curve $\tilde{\C_{h}}$ map to $S_{j}$. Since the eigenvalues of
$\K_{j}$ are pure imaginary, the value of $\mu$ at the points $P_{j}^{\pm}$
is also pure imaginary. On $\C_{h}$ and $\tilde{\C_{h}}$ there is
a natural real structure $J$ induced by the conjugation of $(\lambda,\mu)\in\CC^{2}$.
The points $P_{j}^{\pm}$ that are glued in the singular point $S_{j}$
form a conjugate pair $P_{j}^{+}=JP_{j}^{-}$. If we follow the argument
in \cite{audin:CMP:2002} we can find that the real structure of the
fiber $(\CC^{*})^{r}$ in the extension $(\CC^{*})^{r}\to\jac(\C_{h})\to\jac(\tilde{\C_{h}})$
is given by the maps 
\[
z_{j}\mapsto\frac{1}{\bar{z}_{j}}
\]
 and that the real part of $(\CC^{*})^{r}$ is the torus $(S^{1})^{r}$
given by $z_{j}\bar{z}_{j}=1$. 

\end{proof}
The above theorems can be used to describe the structure of the degenerate
Neumann system as a superintegrable system. The set of local actions
that generates the invariant tori of the Neumann system depends only
on the spectral invariants $F_{1},\ldots,F_{[k-\frac{r}{2}]-1},K_{1}^{2},\ldots,K_{r}^{2}$,
that come from the hypereliptic part of the spectral curve. Apart
from the spectral invariants, there is another set of first integrals
given by the functions of matrices $\K_{j}\in\mathfrak{so}(m_{j})$
that are not the functions of the total angular momenta $K_{j}^{2}$.
The latter first integrals correspond to $O(m_{j})$ symmetries of
the degenerate Neumann system and their Hamiltonian flows are transversal
to the invariant tori. 
\begin{acknowledgements}
I would like to thank to the \emph{Institut de Mathématiques de Toulouse,
}where part of this work has been done. I would also like to thank
Pavle Saksida for carefully reading the manuscript and for his comments,
suggestions and support. This research was partly financed by the
\emph{Slovenian Research Agency (ARRS)}, J1-2152.
\end{acknowledgements}
\bibliography{martin}
 
\end{document}